\documentclass[a4,12pt]{amsart}
\usepackage{amssymb,amstext,amscd,amsfonts,mathdots}
\usepackage{hyperref}
\usepackage{setspace}
\usepackage[all]{xy}
\usepackage{amsthm}
\usepackage{enumerate}
\usepackage{diagbox}
\usepackage{indentfirst}
\usepackage{color}
\usepackage{colortbl}
\usepackage{amsmath}
\newtheorem{theorem}{Theorem}[section]
\newtheorem{theorem*}{Theorem}
\newtheorem{corollary}[theorem]{Corollary}
\newtheorem{corollary*}[theorem*]{Corollary}
\newtheorem{lemma}[theorem]{Lemma}
\newtheorem{proposition}[theorem]{Proposition}

\theoremstyle{definition}
\newtheorem{definition}[theorem]{Definition}

\newtheorem{question}[theorem]{Question}
\newtheorem*{question*}{Question}

\newtheorem*{conjecture*}{Conjecture}
\newtheorem{example}[theorem]{Example}

\newtheorem*{notation*}{Notation}

\newtheorem*{claim*}{Claim}
\newtheorem{definitiontheorem}[theorem]{Definition-Theorem}

\begin{document}
\setlength{\baselineskip}{16pt}

\title[On $\tau$-tilting modules over trivial extensions]
{On $\tau$-tilting modules over trivial extensions of gentle tree algebras}
\author[Qi Wang]{Qi Wang}
\address{Yau Mathematical Sciences Center, Tsinghua University, Beijing 100084, China}
\email{infinite-wang@outlook.com \& infinite-wang@tsinghua.edu.cn}
\author[Yingying Zhang]{Yingying Zhang*}
\address{Department of Mathematics, Huzhou University, Huzhou 313000, Zhejiang Province, China}
\email{yyzhang@zjhu.edu.cn}

\keywords{$\tau$-tilting modules, gentle algebras, trivial extensions, Brauer tree algebras}
\thanks{2020 {\em Mathematics Subject Classification.} 16G20, 16G60}
\thanks{*Corresponding author}

\begin{abstract}
We show that trivial extensions of gentle tree algebras are exactly Brauer tree algebras without exceptional vertex. 
We also give a characterization for the algebras whose trivial extensions are Brauer line/star/cycle algebras. 
As a consequence, the number of support $\tau$-tilting modules over the trivial extension $T(A)$ of a gentle tree algebra $A$ depends only on the number of simple $A$-modules.
\end{abstract}
\maketitle
\section{Introduction}
Support $\tau$-tilting modules were introduced by Adachi, Iyama and Reiten in \cite{AIR}, which can be regarded as a generalization of classical tilting modules from the viewpoint of mutation. It has recently been shown by many mathematicians that the class of support $\tau$-tilting modules admits a nice behavior. For example, the class admits bijections to the class of two-term silting complexes, the class of left finite semibricks and the class of functorially finite torsion classes, see \cite{AIR, Asai} for more details.

One of the combinatorial properties of a finite-dimensional algebra $A$ is the number of support $\tau$-tilting modules over $A$, and $A$ is called \emph{$\tau$-tilting finite} if the number is finite. The number has already been known for several classes of $\tau$-tilting finite algebras, such as, path algebras of Dynkin quivers \cite{ONFR}, Nakayama algebras \cite{Ada-nakayama, GS-Nakayama}, preprojective algebras of Dynkin type \cite{Mizuno}, Brauer tree algebras \cite{AMN-Brauer-tree, Aoki-Brauer-tree}, Brauer cycle algebras with odd cycle length \cite{Aoki-sign-decomposition}, symmetric radical cube zero algebras \cite{AA}, tame Hecke algebras \cite{W-two-point}, tame Schur algebras \cite{W-schur} and so on. Also, one may use the String Applet \cite{G-string applet} to calculate the number of support $\tau$-tilting modules over special biserial algebras, string algebras and gentle algebras. In the cases of path algebras of Dynkin quivers and Brauer tree algebras, one may observe that the number of support $\tau$-tilting modules is independent of the orientations associated with the quivers.

The trivial extension of algebras can be traced back to the trivial extension of abelian categories introduced in \cite{FGR-trivial-extension}. We recall that the trivial extension of an algebra $A$ by an $A$-$A$-bimodule $M$ is defined as $A\ltimes M$ whose underlying additive abelian group is $A\oplus M$ with multiplication given by $(a,n)\cdot (b,m):=(ab,am+nb)$. There are already some interesting results on support $\tau$-tilting modules over $A\ltimes M$, see \cite{LZ-trivial-extension}. In particular, the second author \cite{Zhang-gluing} provided an explicit construction of support $\tau$-tilting modules over triangular matrix algebras, i.e., $(A\oplus B)\ltimes M$ with an $A$-$B$-bimodule $M$.

Let $K$ be an algebraically closed field and $D(-):=\mathsf{Hom}_K(-,K)$. We are interested in the \emph{trivial extension $T(A)$ of $A$} by its minimal co-generator $DA$, i.e., $T(A)=A\ltimes DA$. This special case gives a subclass of symmetric algebras and has attracted a lot of attention, for example, see \cite{AHR-trivial-extension, HRS-iterated-tilted, HW-trivial-extension, Tachikawa-trivial-extension-hereditary}. We simply call $T(A)$ the trivial extension of $A$.

Let $A=K\Delta $ be the path algebra of a Dynkin quiver $\Delta $. Then $T(A)$ is representation-finite (see \cite{Tachikawa-trivial-extension-hereditary}) and obviously, $\tau$-tilting finite, but the number of support $\tau$-tilting modules over $T(A)$ is not known. We note that $A$ (so is $T(A)$) is $\tau$-tilting infinite if $\Delta$ is not of Dynkin type (see, for example, \cite[Theorem 2.6]{Ada-rad-square-0}). In 2016, Iyama pointed out to the second author the following natural question.
\begin{question}\label{question}
Is the number of support $\tau$-tilting modules over $T(A)$ independent of the orientation of $A$?
\end{question}

Actually, the answer is usually negative if $A$ is of type $\mathbb{D}$. For example, we consider path algebras $KQ_i$ ($i=1,2$) with
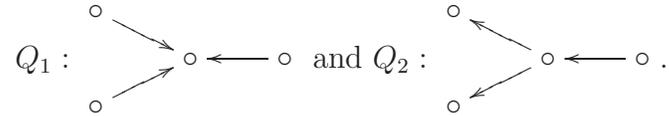
\begin{center}
$Q_1:\vcenter{\xymatrix@C=0.8cm@R=0.2cm{\circ\ar[dr]& &\\
&\circ&\circ\ar[l]\\
\circ\ar[ur]&&}}\quad$ and $\quad Q_2: \vcenter{\xymatrix@C=0.8cm@R=0.2cm{\circ& &\\
&\circ\ar[ul]\ar[dl]&\circ\ar[l]\\
\circ&&}}$.
\end{center}
Then, by direct calculation, the number of support $\tau$-tilting modules over $T(KQ_1)$ is 84 while the number over $T(KQ_2)$ is 86. In this article, we mainly focus on the trivial extensions of path algebras of type $\mathbb{A}$.

Instead of giving an answer for Question \ref{question} directly, we consider a more general setting. A crucial idea we use is the connection between gentle algebras and Brauer graph algebras, that is, the trivial extensions of gentle algebras are exactly the Brauer graph algebras whose multiplicity function identically equals one (see \cite{Schroll-trivial-extension}). Since both gentle algebras and Brauer graph algebras are well-studied in the aspect of $\tau$-tilting theory (\cite{Aoki-Brauer-tree, AAC-Brauer-graph-alg, AMN-Brauer-tree, BDMTY-gentle, FGLZ-gentle, P-gentle}, etc.),  the following result is meaningful.

\begin{theorem}[{Theorem \ref{main-result} and Corollary \ref{cor-main-result}}]
Let $A$ be a finite-dimensional algebra and $T(A)$ the trivial extension of $A$. Then, the following conditions are equivalent:
\begin{enumerate}
  \item $A$ is a gentle tree algebra.
  \item $T(A)$ is a Brauer tree algebra without exceptional vertex.
\end{enumerate}
\end{theorem}

By combining the results in \cite{AMN-Brauer-tree} or \cite{Aoki-Brauer-tree}, it is obvious that if $A$ is a gentle tree algebra with $n$ simples, then the number of support $\tau$-tilting modules over $T(A)$ is $\binom{2n}{n}$.
We can give a positive answer for type $\mathbb{A}$ in Question \ref{question} as follows.
\begin{corollary}[Corollary \ref{answer-to-question}]
Let $A=KQ$ be a path algebra of type $\mathbb{A}$. Then, the number of support $\tau$-tilting modules over $T(A)$ is independent of the orientation of $Q$.
\end{corollary}

We notice that Brauer cycle algebras are also distinguished in the class of Brauer graph algebras. Thus, we also determine the condition for the trivial extension $T(A)$ of a gentle algebra $A$ being a Brauer cycle algebra, as follows.
\begin{theorem}[{Theorem \ref{result-Brauer-cycle} and Corollary \ref{cor-Brauer-cycle}}]
Let $A$ be a finite-dimensional algebra and $T(A)$ the trivial extension of $A$. Then, the following conditions are equivalent:
\begin{enumerate}
  \item $A$ is a gentle algebra with a quiver of type $\widetilde{\mathbb{A}}$, and $\mathsf{rad}^2\ A=0$.
  \item $T(A)$ is a Brauer cycle algebra.
\end{enumerate}
\end{theorem}

This paper is organized as follows. In Section \ref{section-2}, we review the background of $\tau$-tilting theory, Brauer graph algebras and gentle algebras; many useful propositions are also displayed there. In Section \ref{section-3}, we give our main result on gentle tree algebras. Furthermore, we give refined conditions on gentle tree algebras for which the trivial extensions are Brauer line/star algebras. We show the result about Brauer cycle algebras at the end of Section \ref{section-3}.

\section{Preliminaries}\label{section-2}
In this section, we recall the fundamental definitions in $\tau$-tilting theory and collect some results on Brauer graph algebras and gentle algebras which are needed for the later section. For more materials related to $\tau$-tilting theory, we refer to \cite{AIR, AAC-Brauer-graph-alg, P-gentle}, etc.

\subsection{$\tau$-tilting theory}\label{subsection-2.1}
Let $A$ be a finite-dimensional basic algebra over an algebraically closed field $K$. We denote by $\mathsf{mod}\ A$ the category of finitely generated right $A$-modules and by $\mathsf{proj}\ A$ the category of finitely generated projective right $A$-modules. For any $M\in \mathsf{mod}\ A$, let $\mathsf{Fac}(M)$ be the full subcategory of $\mathsf{mod}\ A$ whose objects are factor modules of finite direct sums of copies of $M$. We denote by $A^{\mathsf{op}}$ the opposite algebra of $A$ and by $|M|$ the number of isomorphism classes of indecomposable direct summands of $M$. Let $\tau$ be the Auslander-Reiten translation in $\mathsf{mod}\ A$.

\begin{definition}[{\cite[Definition 0.1]{AIR}}]\label{def-tau-tilting}
Let $M$ be a right $A$-module.
\begin{enumerate}
  \item $M$ is called \emph{$\tau$-rigid} if $\mathsf{Hom}_A(M,\tau M)=0$.
  \item $M$ is called \emph{$\tau$-tilting} if $M$ is $\tau$-rigid and $\left | M \right |=\left | A \right |$.
  \item $M$ is called \emph{support $\tau$-tilting} if there exists an idempotent $e$ of $A$ such that $M$ is a $\tau$-tilting $\left ( A/AeA \right )$-module.
\end{enumerate}
\end{definition}

\begin{definition}[{\cite[Definition 1.1]{DIJ-tau-tilting-finite}}]\label{def-tau-tilting-finite}
An algebra $A$ is called \emph{$\tau$-tilting finite} if it has only finitely many pairwise non-isomorphic basic $\tau$-tilting modules. Otherwise, $A$ is called \emph{$\tau$-tilting infinite}.
\end{definition}

We denote by $\mathsf{i\tau\text{-}rigid}\ A$ (respectively, $\mathsf{s\tau\text{-}tilt}\ A$) the set of isomorphism classes of indecomposable $\tau$-rigid (respectively, basic support $\tau$-tilting) $A$-modules. It is shown by \cite[Theorem 0.2]{AIR} that any $\tau$-rigid $A$-module is a direct summand of some $\tau$-tilting $A$-modules. Moreover, we have the following statement.

\begin{proposition}[{\cite[Corollary 2.9]{DIJ-tau-tilting-finite}}]
An algebra $A$ is $\tau$-tilting finite if and only if one of (equivalently, any of) the sets $\mathsf{i\tau\text{-}rigid}\ A$ and  $\mathsf{s\tau\text{-}tilt}\ A$ is a finite set.
\end{proposition}

One of the important properties of support $\tau$-tilting modules is that the set $\mathsf{s\tau\text{-}tilt}\ A$ admits a partial order $\leq $, which is compatible with the left/right mutations. For any $M,N\in \mathsf{s\tau\text{-}tilt}\ A$, we say $N\leq M$ if $\mathsf{Fac}(N) \subseteq \mathsf{Fac}(M)$. Then, we have

\begin{definitiontheorem}[{\cite[Theorem 2.33]{AIR}}]
For any $M,N\in \mathsf{s\tau\text{-}tilt}\ A$, the following conditions are equivalent:
\begin{enumerate}
  \item $N$ is a left mutation of $M$.
  \item $M$ is a right mutation of $N$.
  \item $M>N$ and there is no $T\in \mathsf{s\tau\text{-}tilt}\ A$ such that $M>T>N$.
\end{enumerate}
\end{definitiontheorem}

It is clear that the Hasse quiver $\mathcal{H}(\mathsf{s\tau\text{-}tilt}\ A)$ of the poset $\mathsf{s\tau\text{-}tilt}\ A$ is constructed by drawing an arrow from $M$ to $N$ if $N$ is a left mutation of $M$. We have known from \cite[Corollary 2.38]{AIR} that $\mathcal{H}(\mathsf{s\tau\text{-}tilt}\ A)$ will exhaust all support $\tau$-tilting modules if it contains a finite connected component. If this is the case, then $A$ is $\tau$-tilting finite.

\subsection{Brauer graph algebras}
We follow the constructions in \cite{Schroll-Brauer-graph}. A \emph{Brauer graph} $G=(G_0,G_1,m,\mathfrak{o})$ is a finite (undirected) connected graph $(G_0,G_1)$ equipped with a multiplicity function $m: G_0\rightarrow \mathbb{Z}_{>0}$ to each vertex in $G_0$, and a cyclic ordering $\mathfrak{o}$ of the edges in $G_1$ around each vertex. By a cyclic ordering here, we have the following:
\begin{itemize}
  \item suppose $v\in G_0$ is a vertex incident to multiple edges $E_1, E_2, \cdots, E_s \in G_1$. Then, the cyclic ordering around vertex $v$ is given by $E_{i_1}<E_{i_2}<\cdots<E_{i_s}<E_{i_1}$, for $i_1,i_2,\cdots,i_s\in \{1,2,\cdots, s\}$;
  \item suppose $v\in G_0$ is a vertex incident to a single edge $E\in G_1$. Then, the cyclic ordering around vertex $v$ is given by $E$ if $m(v)=1$, and by $E<E$ if $m(v)>1$.
\end{itemize}
A \emph{Brauer tree} is a Brauer graph $G$ such that $(G_0,G_1)$ is a tree and $m(v)=1$ for all but at most one vertex $v\in G_0$. If such a vertex exists, then it is called an \emph{exceptional vertex} of the Brauer tree.

We may associate a quiver $Q_G$ with a Brauer graph $G$, where the vertices in $Q_G$ correspond to the edges in $G$ and the arrows in $Q_G$ are induced by the cyclic ordering $\mathfrak{o}$. Let $i$ and $j$ be two distinct vertices in $Q_G$ corresponding to edges $E_i$ and $E_j$ in $G$. Then, there is an arrow $i\overset{\alpha}{\longrightarrow }j$ in $Q_G$ if $E_i<E_j$ appears in the cyclic ordering around their common vertex in $G$, and there is no other $E_k$ satisfying $E_i<E_k<E_j$. In the case of $i=j$, $i\overset{\alpha}{\longrightarrow }i$ is exactly a loop in $Q_G$. It turns out that each vertex $v\in G_0$ gives rise to an oriented cycle in $Q_G$ unless $v$ is a vertex incident to a single edge and $m(v)=1$. By the construction of $Q_G$, there are at most two arrows starting (or ending) at each vertex $i$ in $Q_G$, and $i$ is included at most two oriented cycles in $Q_G$.

We define an admissible ideal $I_G$ of $KQ_G$, which is generated by the following relations:
\begin{itemize}
  \item $C_v^{m(v)}-C_{v'}^{m(v')}$, where $C_v$ and $C_{v'}$ are two distinct oriented cycles in $Q_G$ starting at a vertex $i\in Q_G$, $v$ and $v'$ are vertices in $G$ incident to the edge $E_i$ in $G$;
  \item $C_v^{m(v)}\alpha_1$, where $C_v=\alpha_1\alpha_2\cdots \alpha_k$ is any oriented cycle in $Q_G$ starting at a vertex $i\in Q_G$, $v$ is one of vertices in $G$ incident to the edge $E_i$ in $G$;
  \item $\alpha\beta$, where $\alpha\beta$ is not contained in any oriented cycle in $Q_G$.
\end{itemize}
Then, the bound quiver algebra $B_G=KQ_G/I_G$ is called the \emph{Brauer graph algebra} associated with Brauer graph $G$.

It is well-known that a Brauer graph algebra $B_G$ is exactly a symmetric special biserial algebra, and $B_G$ is of finite representation type if and only if $G$ is a Brauer tree. We may look at some examples.
\begin{example}
Let $G=(G_0,G_1,m,\mathfrak{o})$ be a Brauer graph and $B_G$ its Brauer graph algebra. We set $m(v)=1$ for any vertex $v\in G_0$ (in the sense of both Proposition \ref{prop-Brauer-graph} and Proposition \ref{prop-trivial-extension-of-gentle-alg} below). We assume that the cyclic ordering $\mathfrak{o}$ around each vertex is taken clockwise. Then,
\begin{enumerate}
  \item $B_G$ is called a \emph{Brauer line algebra} if $(G_0,G_1)$ is of the form
       \begin{center}
       $\xymatrix@C=1cm@R=1cm{\circ\ar@{-}[r]^{1}&\circ\ar@{-}[r]^{2}&\circ\ar@{-}[r]^{3}&\cdots \ar@{-}[r]^{n-1} &\circ\ar@{-}[r]^{n}&\circ}$.
       \end{center}
  In this case, $B_G=KQ_G/I_G$ is presented by
       \begin{center}
       $Q_G: \xymatrix@C=1.2cm@R=0.3cm{1\ar@<0.5ex>[r]^{\alpha_1}&2\ar@<0.5ex>[l]^{\beta_1}\ar@<0.5ex>[r]^{\alpha_2}&\cdots
       \ar@<0.5ex>[l]^{\beta_2}\ar@<0.5ex>[r]^{\alpha_{n-2}}&n-1\ar@<0.5ex>[l]^{\beta_{n-2}}\ar@<0.5ex>[r]^{\ \ \alpha_{n-1}}&n\ar@<0.5ex>[l]^{\ \ \beta_{n-1}}}$, and
       \end{center}
       \begin{center}
       $I_G: \left \langle \alpha_i\alpha_{i+1},\beta_{i+1}\beta_i, \beta_i\alpha_i-\alpha_{i+1}\beta_{i+1} \mid 1\leqslant i\leqslant n-2 \right \rangle.$
       \end{center}

  \item $B_G$ is called a \emph{Brauer star algebra} if $(G_0,G_1)$ is of the form
       \begin{center}
       $\vcenter{\xymatrix@C=0.8cm@R=1cm{\circ&\cdots &\circ\\
        \circ&\circ\ar@{-}[l]_{1}\ar@{-}[lu]_{2}\ar@{-}[ru]^{n-1}\ar@{-}[r]^{n}&\circ}}$.
       \end{center}
  In this case, $B_G=KQ_G/I_G$ is presented by
       \begin{center}
       $Q_G: \vcenter{\xymatrix@C=0.7cm@R=0.3cm{ &2\ar[r]^{\alpha}&3\ar[r]^{\alpha}&\cdots \ar[r]^{\alpha}&\circ\ar[rd]^{\alpha} & \\
       1\ar[ru]^{\alpha}&&&&&\circ\ar[dl]^{\alpha} \\
       &n\ar[lu]^{\alpha}&n-1 \ar[l]^{\alpha}&\cdots \ar[l]^{\alpha} &\circ\ar[l]^{\alpha}  & }}$ and $I_G: \left \langle \alpha^{n+1} \right \rangle$.
       \end{center}
       Note that $B_G$ is a symmetric Nakayama algebra, and any symmetric Nakayama algebra is obtained in this way.

  \item $B_G$ is called a \emph{Brauer cycle algebra} if $(G_0,G_1)$ is of the form
       \begin{center}
       $\vcenter{\xymatrix@C=1cm@R=0.3cm{ &\circ\ar@{-}[r]^{2}&\circ\ar@{-}[r]^{3}&\cdots \ar@{-}[r] &\circ\ar@{-}[rd] & \\
       \circ\ar@{-}[ru]^{1}\ar@{-}[rd]^{n}&&&&&\circ\\
       &\circ\ar@{-}[r]^{n-1}&\circ\ar@{-}[r]^{n-2}&\cdots \ar@{-}[r] &\circ\ar@{-}[ru] & }}$.
       \end{center}
  In this case, $B_G=KQ_G/I_G$ is presented by
       \begin{center}
       $Q_G: \vcenter{\xymatrix@C=1.2cm@R=0.5cm{ &2\ar@<0.5ex>[r]^{\alpha_2}\ar@<0.5ex>[ld]^{\beta_1}&3\ar@<0.5ex>[l]^{\beta_2}\ar@<0.5ex>[r]^{\alpha_3}&\cdots \ar@<0.5ex>[l]^{\beta_3}\ar@<0.5ex>[r]&\circ\ar@<0.5ex>[l] \ar@<0.5ex>[rd] & \\
       1\ar@<0.5ex>[ru]^{\alpha_1}\ar@<0.5ex>[rd]^{\beta_n}&&&&&\circ\ar@<0.5ex>[ld]\ar@<0.5ex>[lu] \\
       &n\ar@<0.5ex>[lu]^{\alpha_n}\ar@<0.5ex>[r]^{\beta_{n-1}}&n-1\ar@<0.5ex>[r] \ar@<0.5ex>[l]^{\alpha_{n-1}}&\cdots \ar@<0.5ex>[l] \ar@<0.5ex>[r] &\circ\ar@<0.5ex>[l]\ar@<0.5ex>[ru]  & }}$, and
       \end{center}
       \begin{center}
       $I_G: \left \langle \alpha_n\alpha_1, \alpha_i\alpha_{i+1},\beta_1\beta_n,\beta_{i+1}\beta_i, \beta_n\alpha_n-\alpha_1\beta_1, \beta_i\alpha_i-\alpha_{i+1}\beta_{i+1} \mid 1\leqslant i\leqslant n-1 \right \rangle.$
       \end{center}
\end{enumerate}
\end{example}

We recall some results on Brauer graph algebras which are related to $\tau$-tilting theory.
\begin{proposition}[{\cite[Theorem 1.1]{AAC-Brauer-graph-alg}}]\label{prop-Brauer-graph}
Let $G=(G_0,G_1,m,\mathfrak{o})$ be a Brauer graph and $B_G$ its Brauer graph algebra. Then, we have
\begin{enumerate}
  \item $B_G$ is $\tau$-tilting finite if and only if $G$ contains at most one cycle of odd length and no cycle of even length.
  \item $\mathsf{s\tau\text{-}tilt}\ B_G\simeq \mathsf{s\tau\text{-}tilt}\ B_{G'}$ as posets if $(G_0,G_1,\mathfrak{o})$ and $({G'}_0,{G'}_1, \mathfrak{o}')$ coincide.
\end{enumerate}
\end{proposition}

Although the $\tau$-tilting finiteness of a Brauer graph algebra $B_G$ is known, the number $\#\mathsf{s\tau\text{-}tilt}\ B_G$ for a $\tau$-tilting finite algebra $B_G$ is not clear in general. We only know the number $\#\mathsf{s\tau\text{-}tilt}\ B_G$ for Brauer tree algebras and $\tau$-tilting finite Brauer cycle algebras, as displayed below.

\begin{proposition}[{\cite[Theorem 1.1]{AMN-Brauer-tree} or \cite[Theorem 1.1]{Aoki-Brauer-tree}}]\label{prop-Brauer-tree}
Let $G$ be a Brauer tree with $n$ edges and $B_G$ its Brauer tree algebra. Then, we have $\#\mathsf{s\tau\text{-}tilt}\ B_G=\binom{2n}{n}$.
\end{proposition}

\begin{proposition}[{\cite[Theorem 5.12]{Aoki-sign-decomposition}}]\label{prop-Brauer-cycle}
Let $B_G$ be a Brauer cycle algebra with $n$ simples. If $n$ is odd, then we have $\#\mathsf{s\tau\text{-}tilt}\ B_G=2^{2n-1}$.
\end{proposition}

\subsection{Gentle algebras}\label{subsection-2.3}

We recall the definition of gentle algebras from \cite{AS-tilting-cotilting-equ}. A finite-dimensional algebra $A$ is called a \emph{gentle algebra} if it is Morita equivalent to a bound quiver algebra $KQ/I$, where the following conditions hold:
\begin{itemize}
\item $Q=(Q_0,Q_1,s,t)$ is a finite quiver such that for every vertex $i$ in $Q_0$, there are at most two arrows ending at $i$ and at most two arrows starting at $i$;
\item for each arrow $\alpha$ in $Q_1$, there is at most one arrow $\beta$ such that $t(\alpha)=s(\beta)$ and $\alpha\beta\in I$, and there is at most one arrow $\gamma$ such that $t(\gamma)=s(\alpha)$ and $\gamma\alpha\in I$;
\item for each arrow $\alpha$ in $Q_1$, there is at most one arrow $\beta'$ such that $t(\alpha)=s(\beta')$ and $\alpha\beta'\not\in I$, and there is at most one arrow $\gamma'$ such that $t(\gamma')=s(\alpha)$ and $\gamma'\alpha\not\in I$;
\item $I$ is an admissible ideal of $KQ$ generated by a set of paths of length 2.
\end{itemize}
Then, a gentle algebra $A=KQ/I$ is called a \emph{gentle tree algebra} if $Q$ is a tree.

Gentle algebras are well-studied in the aspects of $\tau$-tilting theory. For example, we have known from \cite{P-gentle} or \cite{Mousavand-biserial-alg} that a gentle algebra $A$ is $\tau$-tilting finite if and only if it is of finite representation type; several combinatorial or geometric tools were constructed in \cite{BDMTY-gentle, OPS-gentle, PPP-gentle}, etc. to classify support $\tau$-tilting $A$-modules; the structure of $\mathcal{H}(\mathsf{s\tau\text{-}tilt}\ A)$ has been studied in \cite{FGLZ-gentle} and so on. We point out that the number $\#\mathsf{s\tau\text{-}tilt}\ A$ for a $\tau$-tilting finite gentle algebra $A$ is not obvious, but one may use the String Applet (see \cite{G-string applet}) to calculate it.

We now focus on the trivial extensions of gentle algebras. We recall the constructions in \cite{Schroll-trivial-extension}. Let $A=KQ/I$ be a gentle algebra and $\mathcal{P}$ the set of non-trivial paths of $A$. We define $\mathcal{M}$ to be the set of maximal paths in $\mathcal{P}$, that is, all paths $p\in \mathcal{P}$ such that $\alpha p\not\in \mathcal{P}$ and $p\alpha \not\in \mathcal{P}$ for all arrows $\alpha\in Q$. By the definition of gentle algebras, every arrow $\alpha\in Q$ is contained in a unique maximal path, and every vertex $i\in Q$ lies in at most two maximal paths. Then, we define $\mathcal{M}_0$ to be the set of trivial paths $e_i$ associated with vertices $i\in Q$, where $i$ satisfies one of the following conditions:
\begin{itemize}
\item $i$ is a source and there is a single arrow starting at $i$;
\item $i$ is a sink and there is a single arrow ending at $i$;
\item there is a single arrow $\alpha$ ending at $i$ and a single arrow $\beta$ starting at $i$, and $\alpha\beta\not\in I$.
\end{itemize}
We set $\overline{\mathcal{M}}:=\mathcal{M}\cup \mathcal{M}_0$. Then, every vertex $i\in Q$ lies in exactly two elements in $\overline{\mathcal{M}}$ which are not necessarily distinct.

We may associate a graph $\Gamma_A$ with a gentle algebra $A=KQ/I$, where the vertices in $\Gamma_A$ correspond to the elements in $\overline{\mathcal{M}}$ and the edges in $\Gamma_A$ correspond to the vertices in $Q$. Let $v(p)$ be the vertex in $\Gamma_A$ corresponding to $p\in \overline{\mathcal{M}}$ and $E_i$ the edge in $\Gamma_A$ corresponding to $i\in Q$. As we mentioned above, each vertex $i\in Q$ lies in exactly two elements in $\overline{\mathcal{M}}$, say, $p$ and $q$, then we set $E_i$ as the edge in $\Gamma_A$ connecting $v(p)$ and $v(q)$.

It is shown in \cite[Lemma 3.1]{Schroll-trivial-extension} that each maximal path $p$ in $\mathcal{M}$ gives rise to a linear order of the edges connected to the corresponding vertex $v(p)$ in $\Gamma_A$. More precisely, let $p=i_1\overset{\alpha_1}{\longrightarrow }i_2\overset{\alpha_2}{\longrightarrow }\cdots \overset{\alpha_n}{\longrightarrow }i_{n+1}$ be a maximal path in $\mathcal{M}$, the linear order of all edges in $\Gamma_A$ connected to $v(p)$ is given by $E_{i_1}<E_{i_2}<\cdots<E_{i_{n+1}}$. Note that this linear order can be completed to a cyclic ordering around vertex $v(p)$ in $\Gamma_A$ by adding $E_{i_{n+1}}<E_{i_1}$ at the end. Then, we call $E_{i_1}<E_{i_2}<\cdots<E_{i_{n+1}}<E_{i_1}$ the cyclic ordering induced by $p$. The following result gives a connection between gentle algebras and Brauer graph algebras via trivial extensions.

\begin{proposition}[{\cite[Theorem 1.2]{Schroll-trivial-extension}}]\label{prop-trivial-extension-of-gentle-alg}
Let $A$ be a gentle algebra and $\Gamma_A$ the graph associated with $A$. We denote by $B_{\Gamma_A}$ the Brauer graph algebra defined on $\Gamma_A$ with $m(v)=1$ for every vertex $v\in \Gamma_A$ and with the cyclic ordering induced by maximal paths in $A$. Then, the trivial extension $T(A)$ of $A$ is isomorphic to $B_{\Gamma_A}$.
\end{proposition}

\section{Main Results}\label{section-3}
Let $A$ be a finite-dimensional basic algebra over an algebraically closed field $K$. We recall from \cite{HR-tilted} that an $A$-module $T$ is called a \emph{tilting module} if $|T|=|A|$, $\mathsf{Ext}_A^1(T,T)=0$ and the projective dimension $\mathsf{pd}_A\ T$ of $T$ is at most one. Dually, an $A$-module $T$ is called \emph{cotilting} (see \cite[Section 4.1]{R-tubular}) if it satisfies $\left | T \right |=\left | A \right |$, $\mathsf{Ext}_A^1(T,T)=0$ and the injective dimension $\mathsf{id}_A\ T$ is at most one.

Let $K\Delta$ be a path algebra of a finite quiver $\Delta$ without oriented cycles. An algebra $A$ is called an \emph{iterated tilted algebra of type $\Delta$} (see \cite[(1.4)]{AH-generalized-tilted}) if there exists a sequence of algebras $K\Delta=A_0,A_1,\dots,A_m=A$ and a sequence of tilting modules $T_{A_i}^i$ with $0\leqslant i\leqslant m-1$, such that $A_{i+1}=\mathsf{End}_{A_i}\ T_{A_i}^i$, and $T_{A_i}^i$ is separating (that is, either $\mathsf{Hom}_{A_i}(T_{A_i}^i, M)=0$ or $\mathsf{Ext}_{A_i}^1(T_{A_i}^i, M)=0$ for any indecomposable $A_i$-module $M$). If $m\leqslant 1$, then $A$ is called a \emph{tilted algebra of type $\Delta$}.

There is a characterization of iterated tilted algebras given in \cite{HRS-iterated-tilted}. We recall that two algebras $A$ and $B$ are said to be \emph{tilting-cotilting equivalent} (see \cite{AS-tilting-cotilting-equ}) if there exists a sequence of algebras $A=A_0,A_1,\dots,A_m=B$ and a sequence of tilting or cotilting modules $T_{A_i}^i$ with $0\leqslant i\leqslant m-1$ such that $A_{i+1}=\mathsf{End}_{A_i}\ T_{A_i}^i$. Then, the main result in \cite{HRS-iterated-tilted} states that $A$ is iterated tilted of type $\Delta$ if and only if $A$ is tilting-cotilting equivalent to the path algebra $K\Delta$ (or, if and only if $A$ is derived equivalent to $K\Delta$). Before proving the main result, we need the following observation.

\begin{lemma}\label{lem-gentle-tree=iterated-titled}
An algebra $A$ is a gentle tree algebra if and only if $A$ is an iterated tilted algebra of type $\mathbb{A}_n$ for some $n\in \mathbb{N}$.
\end{lemma}
\begin{proof}
If $A$ is an iterated tilted algebra of type $\mathbb{A}_n$, then we have an explicit characterization in \cite{AH-generalized-tilted} for the quiver and relations of $A$. Consequently, it is not difficult to find that $A$ is a gentle tree algebra.

Conversely, we assume that $A$ is a gentle tree algebra with $n$ simples. By \cite[IX, Proposition 6.7]{ASS}, there exists a sequence of algebras $A=A_0,A_1,\dots,A_m$ and a sequence of (separating) tilting modules $T_{A_i}^i$ with $0\leqslant i\leqslant m-1$ such that
\begin{center}
$A_{i+1}=\mathsf{End}_{A_i}\ T_{A_i}^i$ or $A_{i+1}=(\mathsf{End}_{A_i}\ T_{A_i}^i)^{\mathsf{op}}$,
\end{center}
and $A_m$ is a path algebra of type $\mathbb{A}_n$. Let $D=\mathsf{Hom}_K(-,K)$. Then, we have
\begin{center}
$(\mathsf{End}_{A_i}\ T_{A_i}^i)^{\mathsf{op}}\simeq \mathsf{End}_{A_i}\ D(T_{A_i}^i)$.
\end{center}
It is known by the definition that $T_{A_i}^i$ is a tilting module if and only if $D(T_{A_i}^i)$ is a cotilting module. Hence, we conclude that $A$ is tilting-cotilting equivalent to a path algebra of type $\mathbb{A}_n$. By the main result in \cite{HRS-iterated-tilted}, $A$ is an iterated tilted algebra of type $\mathbb{A}_n$.
\end{proof}

\begin{theorem}\label{main-result}
Let $T(A)$ be the trivial extension of $A$. Then, $A$ is a gentle tree algebra if and only if $T(A)$ is a Brauer tree algebra without exceptional vertex.
\end{theorem}
\begin{proof}
Let $A$ be a gentle tree algebra. We know from Lemma \ref{lem-gentle-tree=iterated-titled} that $A$ is an iterated tilted algebra of type $\mathbb{A}_n$. On the one hand, \cite[Theorem]{AHR-trivial-extension} claims that an algebra $B$ is iterated tilted of type $\mathbb{A}_n$ if and only if $T(B)$ is of finite representation type and of Cartan class $\mathbb{A}_n$. On the other hand, \cite[Theorem]{HW-trivial-extension} claims that a trivial extension $T(B)$ is of finite representation type and of Cartan class $\mathbb{A}_n$ if and only if $T(B)\simeq T(C)$ for a tilted algebra $C$ of type $\mathbb{A}_n$. Thus, without loss of generality, we may assume that $A$ is a tilted algebra of type $\mathbb{A}_n$.

It is shown in \cite[Section 3]{Skowronski-2006} that the class of trivial extensions of tilted algebras of type $\mathbb{A}_n$ coincides with the class of Brauer tree algebras without exceptional vertex.
\end{proof}

We have the following immediate consequences.
\begin{corollary}\label{cor-main-result}
Let $A$ be a gentle tree algebra with $n$ simples. Then, $\#\mathsf{s\tau\text{-}tilt}\ T(A)=\binom{2n}{n}$.
\end{corollary}
\begin{proof}
This follows from Proposition \ref{prop-Brauer-tree} and Theorem \ref{main-result}.
\end{proof}

\begin{corollary}\label{answer-to-question}
Let $A=KQ$ be a path algebra of type $\mathbb{A}$. Then, the number $\#\mathsf{s\tau\text{-}tilt}\ T(A)$ is independent of the orientation of $Q$.
\end{corollary}
\begin{proof}
It is true that $KQ$ is a gentle tree algebra for any choice of the orientation of $Q$.
\end{proof}

\subsection{Brauer line algebras and Brauer star algebras}
As a refinement of Theorem \ref{main-result}, we shall classify gentle tree algebras $A$ according to whether $T(A)$ is a Brauer line algebra or a Brauer star algebra. Let us start with some observations.

\begin{lemma}\label{lem-subquiver-of-type-D}
Let $A=KQ/I$ be a gentle algebra. If $A$ contains a subquiver $Q'$ of type $\mathbb{D}_4$, then $T(A)$ is neither a Brauer line algebra nor a Brauer star algebra.
\end{lemma}
\begin{proof}
A quiver of type $\mathbb{D}_4$ is given by
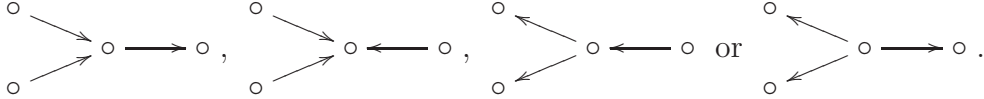
\begin{center}
$\vcenter{\xymatrix@C=0.8cm@R=0.1cm{\circ\ar[dr]& &\\
&\circ\ar[r]&\circ\\
\circ\ar[ur]&&}}$, $\vcenter{\xymatrix@C=0.8cm@R=0.1cm{\circ\ar[dr]& &\\
&\circ&\circ\ar[l]\\
\circ\ar[ur]&&}}$, $\vcenter{\xymatrix@C=0.8cm@R=0.1cm{\circ& &\\
&\circ\ar[ul]\ar[dl]&\circ\ar[l]\\
\circ&&}}$ or $\vcenter{\xymatrix@C=0.8cm@R=0.1cm{\circ& &\\
&\circ\ar[r]\ar[ul]\ar[dl]&\circ\\
\circ&&}}$.
\end{center}
Since $A$ is a gentle algebra, we may suppose (up to opposite algebras) that the subquiver $Q'$ of $A$ is displayed as
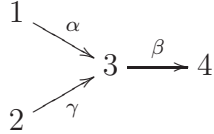
\begin{center}
$\vcenter{\xymatrix@C=0.8cm@R=0.2cm{1\ar[dr]^{\alpha}& &\\
&3\ar[r]^{\beta}&4\\
2\ar[ur]_{\gamma}&&}}$
\end{center}
such that $\gamma\beta\neq 0$ and $\alpha\beta=0$. We have learned from Proposition \ref{prop-trivial-extension-of-gentle-alg} that $T(A)$ is a Brauer graph algebra given by the graph $\Gamma_A$ associated with $A$. Then, we look at the subgraph $\Gamma_{Q'}$ of $\Gamma_A$, which corresponds to the vertices $\{1,2,3,4\}$ in $A$.

Let $\mathcal{M}$ be the set of maximal paths in $A$. Then, it is obvious that $p\alpha, q\gamma\beta t\in \mathcal{M}$ for some paths $p,q,t$ in $A$. Recall that each maximal path $w$ in $\mathcal{M}$ induces a cyclic ordering of the edges connected to the corresponding vertex $v(w)$ in $\Gamma_A$. Hence, the subgraph $\Gamma_{Q'}$ of $\Gamma_A$ is displayed as follows,
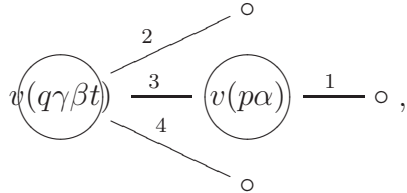
\begin{center}
$\vcenter{\xymatrix@C=0.8cm@R=0.5cm{&\circ&\\
*++[o][F]{v(q\gamma\beta t)}\ar@{-}[r]^{3}\ar@{-}[ru]^{2}\ar@{-}[rd]^{4}&*++[o][F]{v(p\alpha)}\ar@{-}[r]^{\ \ \ 1}&\circ\\
&\circ&}}$,
\end{center}
where the cyclic ordering is taken clockwise and the leaf vertices $\circ$ are not necessarily distinct. This implies that $\Gamma_A$ is neither a Brauer line nor a Brauer star.
\end{proof}

We now assume that $A$ is a gentle tree algebra and $T(A)$ is either a Brauer line algebra or a Brauer star algebra. Then, the above observation states that the quiver of $A$ must be of type $\mathbb{A}$. We define the quiver of type $\mathbb{A}_n$ associated with linear orientation as
\begin{center}
$\vec{A}_n: \xymatrix@C=1cm{1\ar[r]^{\alpha_1}&2\ar[r]^{\alpha_2}&3\ar[r]&\cdots \ar[r]&n-1\ar[r]^{\ \ \alpha_{n-1}}&n}$.
\end{center}

\begin{lemma}\label{lem-linear-quiver}
Let $A=K\vec{A}_n$. Then, the trivial extension $T(A)$ is a Brauer star algebra.
\end{lemma}
\begin{proof}
It is easy to check that
\begin{center}
$\mathcal{M}=\{\alpha_1\alpha_2\cdots\alpha_{n-1}\}$ and $\overline{\mathcal{M}}=\{\alpha_1\alpha_2\cdots\alpha_{n-1}, e_1,e_2,\cdots,e_n\}$.
\end{center}
Then, the graph $\Gamma_A$ is obviously a Brauer star.
\end{proof}

\begin{lemma}\label{lem-rad2-not-zero}
Let $A$ be a gentle tree algebra with a quiver of type $\mathbb{A}_n$ and assume that $A$ is not isomorphic to $K\vec{A}_n$. If $\mathsf{rad}^2\ A\neq 0$, then $T(A)$ is neither a Brauer line algebra nor a Brauer star algebra.
\end{lemma}
\begin{proof}
Let $Q$ be the quiver of $A$. Since $A\not\simeq K\vec{A}_n$ and $\mathsf{rad}^2\ A\neq 0$, $Q$ must contain a subquiver $Q'$ which is either
\begin{center}
$\xymatrix@C=1cm{1\ar[r]^{\alpha}&2\ar[r]^{\beta}&3\ar[r]^{\gamma}&4}$ with $\alpha\beta\neq 0$ and $\beta\gamma=0$,
\end{center}
or
\begin{center}
$\xymatrix@C=1cm{1\ar[r]^{\alpha}&2\ar[r]^{\beta}&3&4\ar[l]_{\gamma}}$ with $\alpha\beta\neq 0$,
\end{center}
up to opposite algebras. Let $\mathcal{M}$ be the set of maximal paths in $A$. Then, one may find that either
\begin{center}
$p\alpha\beta, \gamma q \in \mathcal{M}$ for some paths $p,q$ in $A$,
\end{center}
or
\begin{center}
$p\alpha\beta, q\gamma \in \mathcal{M}$ for some paths $p,q$ in $A$.
\end{center}
By a similar construction with the subgraph $\Gamma_{Q'}$ in the proof of Lemma \ref{lem-subquiver-of-type-D}, we find that $\Gamma_A$ is neither a Brauer line nor a Brauer star.
\end{proof}

\begin{lemma}\label{lem-rad2-zero}
Let $A$ be a gentle tree algebra with a quiver of type $\mathbb{A}_n$. If $\mathsf{rad}^2\ A=0$, then the trivial extension $T(A)$ is a Brauer line algebra.
\end{lemma}
\begin{proof}
Let $Q$ be the quiver of $A$. We may assume that the labeling of $Q$ is given by
\begin{center}
$\xymatrix@C=1cm@R=1cm{1\ar@{-}[r]^{\alpha_1} &2\ar@{-}[r]^{\alpha_2} &3\ar@{-}[r]^{\alpha_3} &\cdots \ar@{-}[r] &n-1\ar@{-}[r]^{\ \ \alpha_{n-1}} &n}$,
\end{center}
and the orientation is arbitrary. Since $\mathsf{rad}^2\ A=0$, we have
\begin{center}
$\mathcal{M}=\{\alpha_1,\alpha_2,\cdots,\alpha_{n-1}\}$ and $\overline{\mathcal{M}}=\{\alpha_1,\alpha_2,\cdots,\alpha_{n-1}, e_1,e_n\}$.
\end{center}
In fact, by the definition of $\mathcal{M}_0$, we have $e_i\not\in \overline{\mathcal{M}}$ for any $1<i<n$. Then, it turns out that  $\Gamma_A$ is a Brauer line and therefore, $T(A)$ is a Brauer line algebra.
\end{proof}

Now, we are able to give the equivalent condition for $T(A)$ being a Brauer star algebra or a Brauer line algebra.
\begin{theorem}
Let $A=KQ/I$ be a gentle tree algebra and $T(A)$ its trivial extension.
\begin{enumerate}
  \item $T(A)$ is a Brauer star algebra if and only if $A\simeq K\vec{A}_n$ for some $n\in \mathbb{N}$.
  \item $T(A)$ is a Brauer line algebra if and only if $\mathsf{rad}^2\ A=0$ and $Q$ is of type $\mathbb{A}$.
\end{enumerate}
\end{theorem}
\begin{proof}
We give a classification for gentle tree algebras as follows.
\begin{itemize}
  \item $A\simeq K\vec{A}_n$;
  \item $Q$ is of type $\mathbb{A}_n$ but $A\not\simeq K\vec{A}_n$;
  \begin{itemize}
    \item $\mathsf{rad}^2\ A=0$;
    \item $\mathsf{rad}^2\ A\neq 0$;
  \end{itemize}
  \item $A$ contains a subquiver of type $\mathbb{D}_4$.
\end{itemize}
It is obvious that this is a complete classification. Then, the statement follows from Lemma \ref{lem-subquiver-of-type-D}, Lemma \ref{lem-linear-quiver}, Lemma \ref{lem-rad2-not-zero} and Lemma \ref{lem-rad2-zero}.
\end{proof}

\subsection{Brauer cycle algebras}
In this subsection, we figure out when the trivial extensions of gentle algebras are Brauer cycle algebras.

We suppose that $A=KQ/I$ is a gentle algebra and $T(A)$ is a Brauer cycle algebra. Then, by our main result Theorem \ref{main-result}, $Q$ is not a tree. By the proof of Lemma \ref{lem-subquiver-of-type-D}, we find that $Q$ cannot contain a subquiver of type $\mathbb{D}_4$. Moreover, a similar proof shows that $Q$ cannot contain a subquiver $Q'$, which is presented by
\begin{center}
$\xymatrix@C=1cm{1\ar[r]^{\alpha}&2\ar@<0.5ex>[r]^{\beta}&3\ar@<0.5ex>[l]^{\gamma}}$ with $\alpha\beta\neq 0, \gamma\beta=0$ or $\alpha\beta=0, \gamma\beta\neq 0$,
\end{center}
or
\begin{center}
$\xymatrix@C=1cm{1\ar[r]^{\alpha}&2\ar@(ur,dr)^{\beta}}$ with $\alpha\beta\neq 0, \beta^2=0$,
\end{center}
up to opposite algebras. Summing these up, $Q$ must be of the form $\widetilde{\mathbb{A}}_n$ for some $n \in \mathbb{N}$.

\begin{theorem}\label{result-Brauer-cycle}
Let $A=KQ/I$ be a gentle algebra. Then, the trivial extension $T(A)$ is a Brauer cycle algebra if and only if $\mathsf{rad}^2\ A=0$ and $Q$ is of type $\widetilde{\mathbb{A}}$.
\end{theorem}
\begin{proof}
We assume that $Q$ is labeled as
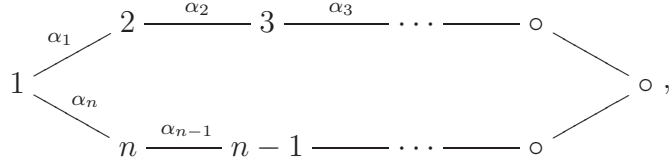
\begin{center}
$\vcenter{\xymatrix@C=1cm@R=0.3cm{ &2\ar@{-}[r]^{\alpha_2}&3\ar@{-}[r]^{\alpha_3}&\cdots \ar@{-}[r] &\circ\ar@{-}[rd] & \\ 1\ar@{-}[ru]^{\alpha_1}\ar@{-}[rd]^{\alpha_n}&&&&&\circ\\
&n\ar@{-}[r]^{\alpha_{n-1}\ \ }&n-1\ar@{-}[r]^{ }&\cdots \ar@{-}[r] &\circ\ar@{-}[ru] & }}$,
\end{center}
with an arbitrary orientation. Recall that $\overline{\mathcal{M}}=\mathcal{M}\cup \mathcal{M}_0$, as defined in subsection \ref{subsection-2.3}.

If $\mathsf{rad}^2\ A=0$, then $\mathcal{M}=\{\alpha_1,\alpha_2,\cdots,\alpha_{n}\}$ and $\mathcal{M}_0=\emptyset$. It turns out that $\Gamma_A$ is a Brauer cycle and hence, $T(A)$ is a Brauer cycle algebra.

If $\mathsf{rad}^2\ A\neq 0$, then there exists at least one $\alpha_i\alpha_{i+1}\notin I$ for some $1\leqslant i\leqslant n$ up to opposite algebras, where we treat $\alpha_{n+1}$ as $\alpha_1$. This implies that $p\alpha_i\alpha_{i+1}q \in \mathcal{M}$ for some paths $p,q$ in $A$, and $e_{i+1} \in \mathcal{M}_0$. It turns out that $\Gamma_A$ contains
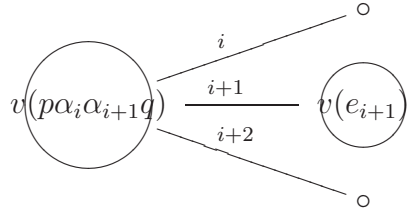
\begin{center}
$\vcenter{\xymatrix@C=1.5cm@R=0.6cm{&\circ\\
*++[o][F]{v(p\alpha_i\alpha_{i+1}q)}\ar@{-}[r]^{i+1}\ar@{-}[ru]^{i}\ar@{-}[rd]^{i+2}&*++[o][F]{v(e_{i+1})}\\
&\circ}}$
\end{center}
as a subgraph. Thus, $\Gamma_A$ is not a Brauer cycle.
\end{proof}

\begin{corollary}\label{cor-Brauer-cycle}
Let $A$ be a gentle algebra with a quiver $Q$ of type $\widetilde{\mathbb{A}}_n$. If $\mathsf{rad}^2\ A=0$ and $n$ is odd, then $\#\mathsf{s\tau\text{-}tilt}\ T(A)=2^{2n-1}$ and it is independent of the orientation of $Q$.
\end{corollary}
\begin{proof}
This follows from Proposition \ref{prop-Brauer-cycle} and Theorem \ref{result-Brauer-cycle}.
\end{proof}

{\bf Conflict of Interest.}
The authors declared that they have no conflicts of interest to this declare.

\section*{Acknowledgements}
The authors thank Professor Osamu Iyama for his insightful suggestions when the second author was visiting Nagoya University. Qi Wang would also like to thank Toshitaka Aoki for many helpful discussions. Qi Wang was partially supported by the JSPS Grant-in-Aid for JSPS Fellows \text{20J10492}. Yingying Zhang was supported by the NSFC (Grant No. 12201211) and the China Scholarship Council (Grant No. 202109710002).


\end{document}